\def\draft{n}
\documentclass[11pt]{amsart}
\usepackage[headings]{fullpage}
\usepackage{amssymb,epic,eepic,epsfig,amsbsy,amsmath,amscd,color}
\usepackage[all]{xy}
\usepackage{graphicx}
\usepackage{texdraw}
\usepackage{url}
\usepackage{bbm}
\usepackage{mathrsfs}
\usepackage{tikz}
\usetikzlibrary{cd}
\usepackage{accents}
\usepackage[normalem]{ulem}

\def\printname#1{
        \if\draft y
                \smash{\makebox[0pt]{\hspace{-0.5in}
                        \raisebox{8pt}{\tt\tiny #1}}}
        \fi
}

\def\lbl#1{\label{#1}\printname{#1}}

                        \theoremstyle{plain}
\newtheorem{theorem}{Theorem}[section]

\newtheorem{lemma}[theorem]{Lemma}
\newtheorem{corollary}[theorem]{Corollary}

\newtheorem{definition}{Definition}

\theoremstyle{definition}

\newtheorem{remark}[theorem]{Remark}

      \def\nc{\newcommand}

   \nc\FI[2]{\begin{figure}
    \begin{center}\input{#1.pstex_t}\end{center}
    \caption{#2}
    \lbl{#1}
  \end{figure}}
\nc\FIG[3]{\begin{figure}
    \includegraphics[#3]{#1.eps}
    \caption{#2}
    \lbl{fig:#1}
    \end{figure}}
\nc\FF[3]{\begin{figure}
    \includegraphics[#3]{#1.eps}
    \caption{#2}
    \lbl{#1}
    \end{figure}}
    \nc\FIGc[3]{\begin{figure}[htpb]
    \includegraphics[height=#3]{#1.eps}
    \caption{#2}
    \lbl{fig:#1}
    \end{figure}}

    \nc\FIGh[3]{\begin{figure}[htpb]
    \includegraphics[height=#3]{#1.eps}
    \caption{#2}
    \lbl{fig:#1}
    \end{figure}}

\def\be { \begin{equation} }
\def\ee { \end{equation} }

\begin{document}

\title{On quotients of congruence subgroups of braid groups}

\author[Jessica Appel]{Jessica Appel}
\address{Department of Mathematics, Indiana University, Bloomington, Indiana 47305, USA}
\email{jlappel@iu.edu}

\author[Wade Bloomquist]{Wade Bloomquist}
\address{School of Mathematics,
 Georgia Institute of Technology, Atlanta, Georgia 30332, USA}
\email{wbloomquist3@gatech.edu}

\author[Katie Gravel]{Katie Gravel}
\address{Department of Mathematics, Massachusetts Institute of Technology, Boston, Massachusetts, 02139, USA}
\email{kgravel@mit.edu}

\author[Annie Holden]{Annie Holden}
\address{Department of Mathematics, University of Notre Dame, Notre Dame, Indiana, 46556, USA}
\email{aholden2@nd.edu}

\thanks{
2020 {\em Mathematics Classification:} Primary 20F36. Secondary 20H05.\\
{\em Key words and phrases: Braid group, Congruence subgroup.}}

\begin{abstract}
The integral Burau representation provides a map from the braid group into a group of integral matrices.  This allows for a definition of congruence subgroups of the braid group as the preimage of the usual principal congruence subgroups of integral matrices.  We explore the structure these congruence subgroups by examining some of the quotients that may arise in the series induced by divisibility of levels.  We build on the work of Stylianakis on symmetric quotients of congruence subgroups, which itself generalizes the quotient of the braid group by the pure braid group.  We accomplish this by utilizing results of Newman on integral matrices and explicitly finding elements in the preimage of any transposition.  Our generalization is made possible by avoiding the use of a generating set for congruence subgroups.  We find further generalizations based on results of Brendle and Margalit as well as Kordek and Margalit on the level four congruence subgroup.  This gives families of quotients which are not isomorphic to symmetric groups.  

\end{abstract}

\maketitle

\section{Introduction}

In this brief note, we use elementary techniques to explore how known results on factor groups in the series of congruence subgroups of braid groups can be generalized.  For each non-negative integer $\ell$, the level $\ell$ congruence subgroup of the braid group, $B_n[\ell]$, is the kernel of the composition of the integral Burau representation with the $\bmod$  $\ell$ reduction map. The integral Burau representation has found uses in many branches of mathematics. For the sake of brevity, we refer the reader to the introduction of \cite{BM} where a list of citations covering applications in algebraic geometry, dynamics, number theory, and topology can be found. Additional details on the integral Burau representation can be found in Section $3$ of Margalit's survey of open problems and questions concerning mapping class groups \cite{Mar}. Our perspective will be based on viewing the integral Burau representation of the braid group as an analog of the symplectic representation of the mapping class group of a surface.  This viewpoint was noted in Section $2.1$ of \cite{BM} for the unreduced Burau representation and is stated here in Lemma \ref{SympConj} for the reduced Burau representation, following the work of \cite{GG}.

Historically, the study of quotients of congruence subgroups of braid groups began with Arnol`d \cite{A}. He showed that the level $2$ congruence subgroup of the braid group may be identified with the pure braid group. Alternatively stated as:
\begin{theorem}[Arnol`d, 1968]
For each $n$,
\[B_n/B_n[2]\cong S_n.\]
\end{theorem}

As a corollary to constructing generators of congruence subgroups of prime level Stylianakis, \cite{Sty1}, proved the following extension:

\begin{theorem}[Stylianakis, 2016]
For each $n$ and each odd prime $p$,
\[B_n[p]/B_n[2p]\cong S_n.\]
\end{theorem}

Without constructing a generating set, we provide an additional extension of this result using results on integral matrices due to Newman \cite{N} and providing an explicit surjection from $B_n[\ell]$ onto $S_n$, namely:

\begin{theorem}[Theorem \ref{symquot}]
For each $n$ and each odd $\ell$,
\[B_n[\ell]/B_n[2\ell]\cong S_n.\]
\end{theorem}

The first generalization of the above results to even levels, can be found in Brendle and Margalit's study of the level $4$ braid group \cite{BM}.  They found the following:
\begin{theorem}[Brendle and Margalit, 2018]
For each $n$,
\[B_n[2]/B_n[4]\cong (\mathbb{Z}/2\mathbb{Z})^{\binom{n}{2}}.\]
\end{theorem}
Again using results on integral matrices and providing the surjection explicitly, we extend as follows:
\begin{theorem}[Theorem \ref{abquot}]
For each $n$ and each odd $\ell$,
\[B_n[2\ell]/B_n[4\ell]\cong (\mathbb{Z}/2\mathbb{Z})^{\binom{n}{2}}.\]
\end{theorem}

In exploring the cohomology of the level $4$ subgroups, Kordek and Margalit, \cite{KM}, proved the following
\begin{theorem}[Kordek and Margalit, 2019]
For each $n$,
\[B_n/B_n[4]\cong \mathcal{Z}_n,\]
where $\mathcal{Z}_n$ is a non-split extension of $S_n$ by $(\mathbb{Z}/2\mathbb{Z})^{\binom{n}{2}}$.
\end{theorem}
We see that taken with this result our Theorems \ref{symquot} and \ref{abquot} imply the following generalization:
\begin{theorem}[Theorem \ref{fivelem}]
For each $n$ and each odd $\ell$,
\[B_n[\ell]/B_n[4\ell]\cong \mathcal{Z}_n.\]

\end{theorem}

A natural direction for future work would be to look into levels which are multiples of higher powers of $2$.  The first step in this direction would be understanding either $B_n/B_n[8]$ or $B_n[4]/B_n[8]$.

\section*{Acknowledgments}  The authors would like to thank Dan Margalit for his helpful conversations and guidance. The first, third and and fourth authors were supported by the $2019$ Georgia Tech REU (funded in part by NSF Grant DMS-17455).  The second author is supported in part by NSF Grant DMS-17455.

\section{Preliminaries}
We begin by recalling the definition of congruence subgroups of braid groups.  Our description of the integral Burau representation will be given explicitly on generators due to the nature of the arguments we will use below. Although this choice may seem to hide some of the rich topological content of this representation, we will describe how that can be recovered in Lemma \ref{SympConj}. 
\begin{definition}
The reduced integral Burau representation is a representation of the $n-$strand braid group, $\rho_{-1}:B_n\rightarrow GL(n-1,\mathbb{Z})$ defined on generators as
\begin{align*}
    & \rho_{-1}(\sigma_1)=\left(\begin{array}{cc}
    1 & 0 \\
    1 & 1
\end{array}\right)\oplus\text{Id}_{n-3}
& \rho_{-1}(\sigma_{n-1})=\text{Id}_{n-3}\oplus \left(\begin{array}{cc}
    1 & -1 \\
    0 & 1
\end{array}\right)
\end{align*}
and for $1<i<n-1$
\[ \rho_{-1}(\sigma_i)=\text{Id}_{i-2}\oplus\left(\begin{array}{ccc}
    1 & -1 & 0 \\
    0 & 1 & 0 \\
    0 & 1 & 1\\
\end{array}\right)\oplus\text{Id}_{n-i-2}.\]

Now recall for integral matrices it is possible to define a natural reduction $\bmod$ $\ell$ map component wise, namely
\[ r_\ell:GL(n-1,\mathbb{Z})\rightarrow GL(n-1,\mathbb{Z}/\ell\mathbb{Z}).\]

Then noting that each generator was sent to a integral matrix the level $\ell$ congruence subgroup of the braid group is the kernel of the composition of these maps, meaning
\[B_n[\ell]:=ker(r_\ell\circ \rho_{-1}).\]
\end{definition}

\begin{remark}
This definition is equivalent to the definition found in \cite{BM}, which uses the unreduced Burau reprsentation.  The equivalence can be seen as the unreduced Burau representation factors into a direct sum of the reduced Burau representation and a $1-$dimensional representation.  
\end{remark}
We now provide a collection of lemmas.
\begin{lemma}[Proposition $2.1$ \cite{GG}]
\label{SympConj}
The reduced integral Burau representation is conjugate to the symplectic representation 
\[\rho:B_{2g+b}\rightarrow \begin{cases}
\mathrm{Sp}(2g,\mathbb{Z}) & b=1\\
(\mathrm{Sp}(2g+2,\mathbb{Z}))_{y_{g+1}} & b=2
\end{cases}\]
coming from the standard action on homology of the surface which double covers the punctured disk, where $(\mathrm{Sp}(2g+2,\mathbb{Z}))_{y_{g+1}}$ is the subgroup of $\mathrm{Sp}(2g+2,\mathbb{Z})$ which fixes the homology class associated one of the two boundary components.

\end{lemma}

\begin{lemma}
\label{power}
For each $m$ and for each $1\leq i\leq n$, we have $\sigma_i^m\in B_n[m]$.
\end{lemma}
\begin{proof}
This follows either from a direct computation of matrix multiplication using the above definition, or as a direct application  of Lemma \ref{SympConj}.  For a discussion of the latter approach see relation ``$R3$: Powers of Dehn Twists" in \cite{Sty1}.
\end{proof}

\begin{corollary}
\label{PurePower}
For each $m$ and for each $1\leq i<j\leq n$ we have $A_{i,j}^m\in B_n[2m]$, where
\[A_{i,j}=\sigma_{j-1}\sigma_{j-2}...\sigma_{i+2}\sigma_{i+1}\sigma_{i}^2\sigma_{i+1}^{-1}\sigma_{i+2}^{-1}...\sigma_{j-2}^{-1}\sigma_{j-1}^{-1}\] are generators of $PB_n$.  
\end{corollary}
\begin{proof}
For any $i$ and $j$, we have that $A_{i,j}$ is a conjugate of $\sigma_i^2$. By Lemma \ref{power}, $\sigma_i^{2m}\in B_n[2m]$. Then as $r_{2m}\circ\rho_{-1}$ is a group homomorphism, the result follows.
\end{proof}

Made relevant in this context by Lemma \ref{SympConj}, we have the following result on integral symplectic matrices due to Newman.
\begin{lemma}[Theorem $VII.22$ \cite{N}]
\label{LemInt}
Let $D=\frac{ab}{GCD(a,b)}$, then 
\[\mathrm{Sp}(2n,\mathbb{Z})[D]=\mathrm{Sp}(2n,\mathbb{Z})[a]\cap \mathrm{Sp}(2n,\mathbb{Z})[b]\]
\end{lemma}
We now apply this to congruence subgroups of the braid group.
\begin{lemma}
\label{LemSplit}
Let $m=\prod_{i}p_i^{\alpha_i}$. Then 
\[\bigcap_i B_n[p_i^{\alpha_i}]= B_n[m],\]
and in particular, for $\ell$ odd we have 
\[B_n[2^{\alpha} \ell]=B_n[2^{\alpha}]\cap B_n[\ell].\]
\end{lemma}
\begin{proof}
The following is a generalization of Lemma $6.2$ in \cite{Sty1}.  As any multiple of $m$ is also a multiple of each $p_i^{\alpha_i}$ we have immediately from the definition that $B_n[m]\leq \bigcap_i B_n[p_i^{\alpha_i}]$.  The reverse inclusion,
$\bigcap_i B_n[p_i^{\alpha_i}]\leq B_n[m]$,
follows from Lemma \ref{SympConj} and repeated applications of Lemma \ref{LemInt}.

\end{proof}

\section{Quotients of Congruence Subgroups}
We may now prove our generalizations as stated in the introduction.  
\begin{theorem}
Let $\ell$ be odd,  then 
\[B_n[\ell]/B_n[2\ell]\cong S_n.\]
\label{symquot}
\end{theorem}
\begin{proof}
The following is a variation of the proof of Theorem $6.1$ given in \cite{Sty1}. We utilize our more general Lemma \ref{LemSplit}, and the use of a generating set for $B_n[\ell]$ is avoided.  

Let $\tau:B_n\rightarrow S_n$ be the natural surjective homomorphism which sends each braid generator $\sigma_i$ to the transposition $(i,i+1)$.  We have that $\ker(\tau)=PB_n= B_n[2]$, as proved by Arnol`d \cite{A}.  Let $\tau_\ell$ be the restriction of $\tau$ to the subgroup $B_n[\ell],$ \begin{align*}\tau_\ell:=\tau|_{B_n[\ell]}:B_n[\ell]\rightarrow S_n.\end{align*}  

As $\ell$ is odd, we have that \begin{align*}\tau(\sigma_i^\ell)=(i,i+1)^{\ell}=(i,i+1),\end{align*} meaning $\tau_\ell$ is surjective utilizing Lemma \ref{power}. 

Finally we have \begin{align*}\ker(\tau_\ell)=\ker(\tau)\cap B_n[\ell]\cong B_n[2]\cap B_n[\ell]\cong B_n[2\ell],\end{align*} where the last isomorphism is a consequence of Lemma \ref{LemSplit}, and our result follows from the first isomorphism theorem.
\end{proof}

\begin{theorem}
\label{abquot}
Let $\ell$ be odd, then 
\[B_n[2\ell]/B_n[4\ell]\cong (\mathbb{Z}/2\mathbb{Z})^{\binom{n}{2}}.\]
\end{theorem}
\begin{proof}
Let \begin{align*}\phi:PB_n\rightarrow H_1(PB_n;\mathbb{Z}/2\mathbb{Z})\cong (\mathbb{Z}/2\mathbb{Z})^{\binom{n}{2}}\end{align*} be the $\bmod$  $2$ abelianization map of the pure braid group.  An explicit description of $\phi$ is given by viewing $(\mathbb{Z}/2\mathbb{Z})^{\binom{n}{2}}$ as the free abelian group on the image of the $\binom{n}{2}$ generators of $PB_n$ which were denoted $A_{i,j}$ in Corollary \ref{PurePower}.  The kernel of $\phi$ is known to be $B_n[4]$ (the main theorem of \cite{BM}).  Now let $\phi_\ell$ be the restriction of $\phi$ to the subgroup $B_n[2\ell]\leq B_n[2]=PB_n$.  From Corollary \ref{PurePower}, we have that $A_{i,j}^\ell\in B_n[2\ell]$.  Now as $\ell$ is odd we have 
\begin{align*}\phi(A_{i,j}^\ell)=\phi(A_{i,j})^\ell=\phi(A_{i,j}),\end{align*} and so $\phi_\ell$ is a surjection.  Additionally utilizing Lemma \ref{LemSplit} we have 
\begin{align*}ker(\phi_\ell)=\ker(\phi)\cap B_n[2\ell]=B_n[4]\cap B_n[2]\cap B_n[\ell]=B_n[4]\cap B_n[\ell]=B_n[4\ell],\end{align*}
where we have also used that $B_n[4]\leq B_n[2]$ and so $B_n[4]\cap B_n[2]=B_n[4].$  Then our desired result follows from the first isomorphism theorem.
\end{proof}

\begin{theorem}
\label{fivelem}
Let $\ell$ be odd, then we have
\[B_n[\ell]/B_n[4\ell]\cong B_n/B_n[4].\]
\end{theorem}
\begin{proof}
The third isomorphism theorem gives us that 
\[(B_n[\ell]/B_n[4\ell])/(B_n[2\ell]/B_n[4\ell])\cong B_n[\ell]/B_n[2\ell],\]
meaning the two rows of the following commutative diagram are exact:
\[\begin{tikzcd}
1\arrow[r] & B_n[2]/B_n[4] \arrow[d]\arrow[r] & B_n/B_n[4]\arrow[r]\arrow[d] & B_n/B_n[2]\arrow[d]\arrow[r] & 1\\
1\arrow[r] & B_n[2\ell]/B_n[4\ell] \arrow[r]& B_n[\ell]/B_n[4\ell] \arrow[r]& B_n[\ell]/B_n[2\ell] \arrow[r]& 1\\
\end{tikzcd}\]
where the vertical arrows are $\ell^{\text{th}}$ power maps used to construct isomorphisms in Theorem \ref{symquot} and Theorem \ref{abquot}.  Then our result follows from an application of the five lemma.  

\end{proof}

\end{document}